\documentclass[
a4paper]{article}

\usepackage{amsfonts,amssymb}
\usepackage{amsthm}
\def\qbox#1{\quad\hbox{#1}}
\def\qqbox#1{\quad\hbox{#1}\quad}
\def\Z{\mathbb Z}

\let\ge\geqslant
\let\^\widehat
\def\gp#1{\left\langle#1\right\rangle}
\def\({\left(}
\def\){\right)}
\let\tmpl\{
\def\{{\left\tmpl}
\let\tmpr\}
\def\}{\right\tmpr}

\newbox\tmpbox
\newdimen\tmpdim
\def\narrow[#1]#2\par %„Ґ« Ґв #2 иЁаЁ­л ­Ґ Ў®«миҐ #1, Ї®¬Ґй Ґв ў \tmpbox
{%
\setbox\tmpbox\hbox{#2}%
\ifdim\wd\tmpbox>#1
\setbox\tmpbox\vbox{\hsize=#1 #2}%
\tmpdim=\ht\tmpbox%
\loop%
\setbox\tmpbox\vbox{\hsize=\wd\tmpbox \advance\hsize by -1pt
#2}%
\ifdim\ht\tmpbox=\tmpdim%
\relax%
\repeat%
\setbox\tmpbox\vbox{\hsize=\wd\tmpbox \advance\hsize by 1pt #2}%
\fi%
\box\tmpbox
}
\def\disp#1{
                $$
                \setbox\tmpbox\vbox{\narrow[\hsize]\noindent#1\par}
                \box\tmpbox
                $$
          }

\newtheorem*{Myasnikov--Romankov Theorem}{Myasnikov--Roman'kov Theorem
{\rm\cite{[MR14]}}}
\newtheorem*{Main Theorem}{Main Theorem}

\newtheorem{Lemma}{Lemma}

\newtheorem{Corollary}{Corollary}
\title{Verbally closed virtually free subgroups
}
\author{Anton A. Klyachko \qquad Andrey M. Mazhuga
}
\date{\small
      Faculty of Mechanics and Mathematics\\
      Moscow State University\\
      Moscow 119991, Leninskie gory, MSU\\
      klyachko@mech.math.msu.su\quad mazhuga.andrew@yandex.ru
}

\begin{document}
\maketitle
\begin{abstract}
\vskip-10mm
\narrower
\narrower
\narrower
\noindent
A theorem of Myasnikov and Roman'kov says that any verbally closed
subgroup of a finitely generated free group is a retract. We
prove that all free (and many virtually free) verbally
closed subgroups are retracts in \emph{any} finitely
generated group.
\end{abstract}

{\def\thefootnote{}\addtocounter{footnote}{-1}%
\footnote{%
This work was supported by the Russian Foundation for Basic Research,
project no. 15-01-05823.} }

%%%%%%%%%%%%%%%%%%%%%%%%%%%%%%%%%%%%%%%%%%%%%%%%%%%%%%%%%%%%%%%%%%%%%%%

\section{Introduction}

A subgroup $H$ of a group $G$ is called \emph{verbally closed}
\cite{[MR14]}
(%
see also
\cite{[Rom12]},
\cite{[RKh13]},
\cite{[Mazh17]})
if
any equation of the form
$$
w(x_1,x_2,\dots)=h,
\qbox{where $w$ is an element of the free group
$F(x_1,x_2,\dots)$ and $h\in H$,
}
$$
having a solution in $G$ has a solution in $H$.
If each finite system of equations with coefficients from $H$
$$
\{w_1(x_1,x_2,\dots)=1, \dots, w_m(x_1,x_2,\dots)=1\},
\qbox{where $w_i\in H*F(x_1,x_2,\dots)$,}
$$
having a solution in $G$ has a solution in $H$, then the subgroup $H$ is
called \emph{algebraically closed} in $G$.

Clearly, any retract (i.e. the image of an endomorphism $\rho$ such
that $\rho\circ\rho=\rho$) is an algebraically closed subgroup.
It is easy to show \cite{[MR14]} that for finitely presented groups
the converse
is also true:
\disp{\hfuzz15pt\sl
A finitely generated
subgroup of a finitely presented group is algebraically closed
if and only if it is a retract.
}%
For the (wider) class of verbally closed subgroups, no similar
structural description is known. However, in free groups, the situation
is simple: verbally closed subgroups, algebraically closed subgroups, and
retracts are the same things.

\begin{Myasnikov--Romankov Theorem}
Verbally closed subgroups of finitely generated free groups are
retracts.
\end{Myasnikov--Romankov Theorem}

A similar fact is valid for  free nilpotent groups
\cite{[RKh13]}.
We generalise the Myasnikov--Roman'kov theorem in two directions:
first, we study subgroups of
arbitrary groups; and secondly, we study not only free subgroups
but also virtually free subgroups $H$,
i.e. containing free subgroups of finite index (in $H$).

\goodbreak

\begin{Main Theorem}
Let $G$ be any group and let $H$ be its
verbally closed
%Є®­Ґз­® Ї®а®¦¤с­­ п  !!!!!!!!
virtually free
infinite
non-dihedral
subgroup containing no
infinite abelian noncyclic subgroups.
Then
\begin{enumerate}
\item[\rm1)]
%Ґб«Ё $H$ ­ҐжЁЄ«ЁзҐбЄ п Ї®¤ЈагЇЇ , в®
$H$ is algebraically closed in $G$;
\item[\rm2)]
if $G$ is finitely generated over $H$
{\rm (i.e. $G=\gp{H,X}$ for some finite subset $X\subseteq G$)},
then $H$ is a
retract of~$G$;
\item[]
in particular, $H$ is finitely generated if
$G$ is finitely generated.
\end{enumerate}
{\rm(For an infinite group,
\emph{non-dihedral} means non-isomorphic to the free product of two
groups of order two.)}
\end{Main Theorem}

\noindent
Note that each
nontrivial
free subgroup $H$ satisfies all conditions of the theorem
and, therefore, is a retract of any finitely generated group
containing $H$ as a verbally closed subgroup.
Even this
corollary seems to be nontrivial.
The following corollary strengthens Theorem 1(1) of~\cite{[Mazh17]}.

\begin{Corollary}
In a free product of finitely many finite groups, any
verbally closed infinite non-dihedral subgroup
is a retract.
\end{Corollary}

In Section 2, we discuss examples showing that the main theorem
cannot be improved in some sense (there are open questions
remaining though). Section 3 contains the proof of the theorem. Our
argument is slightly tricker than that in \cite{[MR14]} but is also based
on the use of Lee words \cite{[Lee02]}.

Let us fix the notation.
If
$k\in \Z$, $x$ and $y$ are elements of a group, then $x^y$, $x^{ky}$,
and $x^{-y}$ denote $y^{-1}xy$, $y^{-1}x^ky$, and $y^{-1}x^{-1}y$,
respectively.
%Љ®¬¬гв в®а~$[x,y]$ ¬л Ї®­Ё¬ Ґ¬ Є Є $x^{-1}y^{-1}xy$. %!!!!!
The commutator subgroup of a group $G$ is denoted by $G'$.
If $X$ is a subset of
a group, then $|X|$, $\gp X$,
%,
%$\nc X$
and $C(X)$ mean the cardinality of $X$, the subgroup generated by $X$,
%,
%­®а¬ «м­®Ґ § ¬лЄ ­ЁҐ ¬­®¦Ґбвў  $X$
and the centraliser of~$X$.
The index of a subgroup $H$ of a group $G$ is denoted by $|G{:}H|$.
%‘Ё¬ў®«~$N(H)$ ®Ў®§­ з Ґв ­®а¬ «Ё§ в®а Ї®¤ЈагЇЇл $H$ (ў ЈагЇЇҐ $G$).
The symbol $A*B$ denotes the free product of groups $A$ and $B$.
$F(x_1,\dots,x_n)$ or $F_n$
is the free group (with a basis $x_1,\dots,x_n$).

%%%%%%%%%%%%%%%%%%%%%%%%%%%%%%%%%%%%%%%%%%

\section{Examples}

Let us see whether it is possible to
omit some conditions of the Main Theorem.

\smallskip\noindent{\bf
The subgroup $H$ is infinite.}
This condition cannot be removed.
Let
$G$ be
the central product of two copies of the quaternion group
(of order eight). Clearly, the factors of this
product are not retracts
(and, therefore, they are not
algebraically closed because, for finitely generated subgroups of finitely
presented groups, this
is the same thing).
Indeed, the kernel of such hypothetical retraction must
be a nontrivial normal subgroup and the group $G$ is nilpotent.
Therefore, this nontrivial normal subgroup must nontrivially
intersect the centre of $G$ (see, e.g., \cite{[KaM82]}). This
leads immediately to a contradiction because the centre in this case
is contained in both factors.

Let us show that (for instance) the second factor $H$ of
$G=Q_8\mathop\times\limits_C Q_8$ (where $C=\{\pm1\}$)
is verbally closed in $G$.
Suppose that an equation
$$
w(x_1,\dots,x_n)=(1,h'),
$$
where $(1,h')\in H$ has a solution
$x_i=(h_i,h_i')$ in
$G$. Let us show that this equation has a solution in
$H$ too.
Suppose that
the sum $k$ of powers of
a variable $x_i$ in
$w(x_1,\dots,x_n)$ is odd.
In this case, the equation $x^k = h'$ has a solution $q$ in $Q_8$
and the substitution
$x_i = (1,q)$ and $x_j = (1,1)$ for $j\ne i$ is
a solution to the initial equation in $H$.

Now, suppose that
all variables occur in $w(x_1,\dots,x_n)$ with
even total powers. In this case, $h'=1$ or $h'=-1$
(otherwise the equation has no solutions in $G$).
If
$h'=1$, then the substitution $x_i=(1,1)$ for all $i$ is a solution
belonging to $H$. If  $h'=-1$, then either
$x_i=(1,h_i)$ or $x_i'=(1,h_i')$
is a solution (lying in $H$).

\smallskip\noindent{\bf
The subgroup $H$ is non-dihedral.}
We do not know whether this condition may be omitted and leave it as an
open question.

\smallskip\noindent{\bf
Any infinite abelian subgroup of $H$ is cyclic.}
A slight modification of the central product considered above shows that
this condition may not be omitted.
Put
$G=Q_8\mathop\times\limits_C Q_8\times F$,
where $F$ is a nontrivial free group.
Consider
the product of the second and third factors as
the subgroup
$H$.
The above argument implies that $H$ is verbally
closed but not a retract
(because a retract of a retract is a retract and the second factor is not
a retract of $G$
and even of $Q_8\mathop\times\limits_C Q_8$
as proved above).

\smallskip\noindent{\bf
The subgroup $H$ is virtually free.}
%It is hard to believe that this (main) condition may be removed. However,
%we know no
%counterexamples and leave it as an open question.
This condition may not be omitted.
Indeed, it is known that the free Burnside group
$B(m,n)$ of rank $m\ge2$ and
odd
exponent $n\ge665$
is infinite while all its abelian subgroups are finite
(see, e.g, Theorems 1.5 and 3.3 in \cite{[Adjan]}).
Put $G=Q_8\mathop\times\limits_C Q_8\times B(2,2017)$
and let $H$ be the product of the second and third factors.
Similarly to the above example, we see that $H$ is verbally closed but not
a retract.

\smallskip\noindent{\bf
The group $G$ is finitely generated over $H$ in assertion 2).}
The following example shows that this condition
may not be omitted.
%12.02.23 
Take the subgroup $H=\gp{1,1,\dots)}$
in the Cartesian (=unrestricted) sum $\Z_2\oplus\Z_3\oplus\Z_5\oplus\dots$
of the prime-order cyclic groups.
Clearly, $H$ is verbally closed
(in abelian
groups, verbally closed subgroups are the same thing as algebraically
closed subgroup, and the same as serving (pure) subgroups).
On the other hand, $H$ is not a direct summand in $G$
(i.e. $H$ is not a retract). Indeed,
if $G=H\oplus D$, then the torsion subgroup $T(G)$ of $G$
(the direct (=restricted) sum of $\Z_p$) must lie in $D$;
but then 
\let\iso\simeq
$H\iso\Z$ is a direct summand in 
$G\iso H\oplus\bigl(D/T(G)\bigr)$,
which is impossible as $G/T(G)$ is a divisible group.%
\footnote{%
The journal version of the paper contains an error in this example.
The authors thank 
Mikhail Mikheenko for pointing this out in 2023.
}

\iffalse %!!!!!!!!
‚®§м¬с¬ Ї®¤ЈагЇЇг $H=\gp1$ жҐ«ле зЁбҐ« ў  ¤¤ЁвЁў­®© ЈагЇЇҐ $G$
жҐ«ле $p$- ¤ЁзҐбЄЁе зЁбҐ«. џб­®, зв® ®­  ўҐаЎ «м­® § ¬Є­гв  (ў  ЎҐ«Ґўле
ЈагЇЇ е ўҐаЎ «м­® § ¬Є­гвлҐ Ї®¤ЈагЇЇл нв® в® ¦Ґ б ¬®Ґ, зв®  «ЈҐЎа ЁзҐбЄЁ
§ ¬Є­гвлҐ Ї®¤ЈагЇЇл, Ё в® ¦Ґ б ¬®Ґ, зв® бҐаў ­в­лҐ (зЁбвлҐ) Ї®¤ЈагЇЇл).
‘ ¤агЈ®© бв®а®­л, ЈагЇЇ  $G$ ­Ґ ¤®ЇгбЄ Ґв ў®®ЎйҐ ­ЁЄ ЄЁе
­ҐваЁўЁ «м­ле а §«®¦Ґ­Ё© ў Їап¬го бг¬¬г (б¬., ­ ЇаЁ¬Ґа, \cite{[Kur67]}),
в® Ґбвм $H$ ­Ґ аҐва Єв.
\fi%!!!!!!!

Take the subgroup $H=\gp1$ of integers in the additive group $G$
of $p$-adic integers. Clearly, $H$ is verbally closed (in abelian
groups verbally closed subgroups are the same thing as algebraically
closed subgroup, and the same as serving (pure) subgroups).
On the other hand, $G$ admits no
nontrivial decompositions into a direct sum (see, e.g., \cite{[Kur67]}),
hence, $H$ is not a retract.

%%%%%%%%%%%%%%%%%%%%%%%%%%%%%%%%%%%%%%%%%%%%%%%%%%%%%%%%%

\section{Proof of the Main Theorem}

Note that any virtually free group $H$ is
linear (even over $\Z$ if $H$ is countable) because the free
group
(of any cardinality)
is linear (over a field)
(see, e.g., \cite{[KaM82]}), and virtually
linear group is also linear. Therefore, $H$ is
\emph{equationally Noetherian}
%(as any linear
%(over field) group
\cite{[BMRe99]}, i.e.
any system of equations with coefficients from $H$
and finitely many unknowns is
equivalent to its finite
subsystem.%
\footnote{Note also that a group containing an equationally Noetherian
subgroup of finite index, is equationally Noetherian
\cite{[BMRo97]}.
}
This, in
turn, means that $H$
is a retract of any finitely
generated over $H$ group containing $H$ as an algebraically closed
subgroup \cite{[MR14]}. Therefore, it suffices to prove assertion 1)
of the theorem.

First, suppose that $H$ is cyclic.
Recall that any integer matrix can be reduced to a diagonal matrix by
integer elementary
transformations.
This means that any finite system of equations over $H$
$$
\{w_1(x_1,x_2,\dots)=1,\ \dots,\ w_m(x_1,x_2,\dots)=1\},
\qbox{where $w_i\in H*F(x_1,x_2,\dots)$,}
$$
can be reduced to the form
$$
\{x_1^{n_1}u_1(x_1,x_2,\dots)=h_1,\ \dots,\
x_m^{n_m}u_m(x_1,x_2,\dots)=h_m\},
\hbox{ where $u_i\in (H*F(x_1,x_2,\dots))'$, $n_i\in\Z$, $h_i\in H$,}
$$
by means of a finite sequence of transformations of the form
$
w_i\to w_iw_j^{\pm1}
\qqbox{and}
x_i\to x_ix_j^{\pm1}.
$
The obtained system has the same number of solutions
(in $G$ or in $H$)
as the initial one.
%!!!
Suppose that this system has a solution in $G$ and
suppose also that each word $u_i$ is a product of~$s$
commutators in $H*F(x_1,x_2,\dots)$. Then each single equation
$x_i^{n_i}[y_1,z_1]\dots[y_s,z_s]=h_i$
(where $y_j$ and $z_j$ are new variables)
has a solution in $G$ and,
therefore, has a solution $\(\^x_i,\^y_1^{(i)},\^z_1^{(i)},\dots\)$
in $H$ (because the subgroup $H$ is verbally
closed). Then,
obviously, $\(\^x_1,\^x_2,\dots\)$ is a solution to the entire system
(since the commutator subgroup of~$H$ is trivial).
%!!!

The case of virtually cyclic subgroup $H$ can be easily reduced to
the case of cyclic $H$ by virtue of the following observation:
\disp{\sl
an infinite virtually cyclic group containing no
infinite noncyclic abelian subgroups is either cyclic or
dihedral.
}%
Indeed, any virtually cyclic group contains a finite
normal subgroup such that the quotient group is either cyclic or
dihedral (see, e.g., \cite{[Sta71]}). This finite normal
subgroup must be trivial because
otherwise we can find in its
centraliser (which is of finite index) an element of infinite
order and obtain an infinite noncyclic abelian subgroup.

\medskip

Now, consider the more difficult case of non-virtually-cyclic group $H$.

\begin{Lemma}\label{Lem2}
In a virtually free group which is not virtually cyclic,
any element decomposes into a product of two infinite-order elements.
\end{Lemma}

\begin{proof}
Clearly, it suffices to prove this assertion for
finitely generated groups.
Recall, that
a finitely generated
virtually free group
admits an
action on a
(directed) tree such that the stabilisers of vertices are finite
\cite{[KPS73]}. Recall also that any fixed-point-free automorphism of
a tree has a unique invariant line (the \emph{axis})~\cite{[Ser77]}.

Consider such an action of a group $H$ on a tree $T$.
Let $h\in H$ be an element we want to decompose
and let $T_h\subseteq T$ be the fixed point set for $h$.
If the order of $h$ is infinite, then $h=h^2h^{-1}$ is the required
decomposition; therefore, we assume that the order of $h$ is finite and
$T_h$ is nonempty (and, hence, connected).

Take some element $g\in H$ of infinite order
with an axis $l_g$.
If~$g^{-1}h$ has no fixed points, then its order
is infinite and $h=g\cdot(g^{-1}h)$ is a required decomposition.
Let $a\in T$ be a fixed point for $g^{-1}h$. Then $h(a)=g(a)=b$.

The equality $h(a)=b$ shows that the
path joining $a$ and $b$ must intersect the
subtree~$T_h$ in a unique point~$c$, and $c$ is the midpoint
of this path.
The equality $g(a)=b$ shows that the line $l_g$ passes through
$c$ and intersects the segments $[a,c]$ and $[c,b]$ in segments
of the same
nonzero length $\delta$ (otherwise, the distances
from~$l_g$ to $a$ and to $b$ are different and, hence, $g(a)\ne b$).
The element $g$ must act as a translation by
$2\delta$ on its axis $l_g$. However, we only need the
equality $T_h\cap l_g=\{c\}$ now.

Take another infinite-order element $g'\in H$ with another
axis~$l_{g'}\ne l_g$ (such an element exists because
the group $H$ is not virtually cyclic). A similar argument shows that
either $h=g'\cdot((g')^{-1}h)$ is the required decomposition or
$T_h\cap l_{g'}=\{c'\}$ for some point $c'$.
In the latter case, take the element
$g''=g^kg'g^{-k}$.
Its axis is $g^k(l_{g'})$ obviously;
and we see that this line does not
intersect the subtree $T_h$ if the number
$k$ is chosen large enough (positive or negative).
Therefore, in this last case,
$h=g''\cdot((g'')^{-1}h)$ is the required decomposition.
(More precisely, if $c'\ne c$, then $k$ can be chosen equal to 1 or
any other nonzero number; if  $c'=c$, then $k$
should be choose in such a way that
$|2k\delta|$
is larger than the distance from $c$
to the endpoint of the segment or half-line $l_g\cap l_{g'}$.)
\end{proof}

\begin{Lemma}\label{Lem4}
If $h_1$ and $h_2$ are infinite-order elements of a
 virtually free group all whose infinite abelian
subgroup are cyclic and $h_1^k=h_2^k$ for some nonzero integer $k$,
then $h_1=h_2$.
\end{Lemma}

\begin{proof}
The
roots of an element lie in its centraliser, therefore, it
suffices to show that the centraliser of an infinite-order element $h$ is
cyclic. This centraliser $C(h)$ is a virtually free group
(as a subgroup of a virtually free group) with infinite centre.
This implies immediately that $C(h)$ is virtually cyclic.
It remains to refer to the fact mentioned in the beginning of this
section: a virtually cyclic group without
infinite
abelian noncyclic subgroups is
either cyclic or dihedral. In the case under consideration, the group
cannot be dihedral because
the centre of the centraliser is nontrivial.
\end{proof}

The following lemma is well known.
\begin{Lemma}\label{LemWA}
If a subgroup $H$ of a group $G$ is such that any finite system
of equations of the
form
$$
\{w_1(x_1,\dots,x_n)=h_1,\ \dots,\  w_m(x_1,\dots,x_n)=h_m\},
\qbox{where $w_i\in F(x_1,\dots,x_n)$ and $h_i\in H$,}
\eqno{(1)}
$$
having a solution in $G$
has a solution in $H$ too, then $H$ is algebraically closed.
\end{Lemma}

\begin{proof}
Just denote the coefficients by new letters and interpret them as
variables. For example, the solvability of the equation
$xyh_1[x^{\the\year},h_2]y^{-1}=1$
is equivalent to the solvability of the system
$$
\{xyz[x^{\the\year},t]y^{-1}=1,\
z=h_1,\
t=h_2\}.
$$
\end{proof}

%\medskip

Now, we need a useful tool.
Recall that a \emph{Lee word}
in $m$ variables
for the free
group of rank $r$
is an element
%\newline
$L(z_1,\dots,z_m)$ of the free group of rank $m$ such that
\goodbreak
\begin{itemize}
\item[1)]
if $L(v_1,\dots,v_m)=L(v_1',\dots,v_m')\ne1$ in $F_r$, then
$v_i'\in F_r$ are obtained from $v_i\in F_r$ by
simultaneous conjugation, i.e.,
there exists $w\in F_r$ such that
$v_i'=v_i^w$ for all $i=1,\dots,m$;

\nobreak
\item[2)]
$L(v_1,\dots,v_m)=1$ if and only if the elements
$v_1,\dots,v_m$ of $F_r$ generate a cyclic subgroup.

\end{itemize}

\goodbreak

In \cite{[Lee02]}, such words were constructed for all
integers $r,m\ge2$.
Actually, it is easy to see that Lee's result implies
the existence of
a \emph{universal Lee word} in $m$ variables.

\begin{Lemma}
For any positive integer $m$, there exists
an element
$L(z_1,\dots,z_m)\in F_m$
such that
properties {\rm 1)} and {\rm 2)} hold in
all free groups
$F_r$ and even in $F_\infty$.
\end{Lemma}

\begin{proof}
This assertion follows immediately from Lee's result and the following
simple fact:
\disp{\sl
$F_\infty$ embeds into $F_2$ as a \emph{malnormal} subgroup,
}%
i.e. a subgroup
$S\subset F_2$
such that $S^f\cap S=\{1\}$ for all $f\in F_2\setminus S$.
This fact
follows, e.g., from a result of~\cite{[Wise01]}:
\disp{\sl
\hfuzz 6pt
\narrower
\narrower
\narrower
\narrower
\narrower
\narrower
\narrower
in a free group, any set
satisfying small-cancellation condition $C(5)$ freely
generates a malnormal subgroup.
}
Thus, a Lee word for $F_2$ is universal, i.e. it is suitable also
for $F_\infty$.
\end{proof}

Let us proceed with the proof of the Main Theorem.
Thus, we assume that a verbally closed subgroup~$H$ of a group $G$ is
virtually free, does not contain
infinite
noncyclic abelian subgroups,
and contains a
normal (in~$H$) non-abelian free subgroup $F$ of index $N$ (in~$H$).
Applying Lemma \ref{LemWA}, we can assume that
system (1) has a solution in $G$ and we have to show
that this system has a solution in $H$.

Let
$L(z_1,\dots,z_{2m+2})$ be a
universal Lee word in $2m+2$ variables.
By Lemma \ref{Lem2}, for each element $h_i$, we find
an infinite-order element $f_i\in H$ such that the order of $h_if_i$
is also infinite. Take two noncommuting element $u_1,u_2\in F$ and
consider the equation
$$
L\Bigl((w_1(x_1,\dots,x_n)y_1)^N,\dots,(w_m(x_1,\dots,x_n)y_m)^N,
y_1^N,\dots,y_m^N,
z_1^N,z_2^N\Bigr)
=f,
$$
where
$f=L\Bigl((h_1f_1)^N,\dots,(h_mf_m)^N,
f_1^N,\dots,f_m^N,
u_1^N,
u_2^N\Bigr)
\in F$.
This equation has a
solution in $G$ by construction
(just take the following:
a solution to system (1) as $x_i$,
elements $f_i$ as $y_i$, and
$u_i$ as $z_i$).

The subgroup $H$
is verbally closed in $G$ and $f\in H$; thus, the last equation has
a solution~$\^x_i,\^y_j,\^z_k$ in $H$.

The right-hand side of the equation is a value of a Lee word on some
elements of the free group $F$ (because $h^N\in F$ for all $h\in H$);
the elements  $u_1^N$ and $u_2^N$ do not commute
(because the elements $u_1$ and $u_2$ of a free group
are
chosen non-commuting), so they do not lie in the same cyclic subgroup and,
therefore,
Property 2) of Lee words implies that $f\ne1$.
According to Property 1),
we have
$$
(w_i(\^x_1,\dots,\^x_n)\^y_i)^{Nw}=(h_if_i)^N,
\quad
\^y_i^{Nw}=f_i^N
\qqbox{and}
\^z_i^{Nw}=u_i^N
\qbox{for some $w\in F$.}
$$
By
Lemma \ref{Lem4}, we obtain the equalities
$$
(w_i(\^x_1,\dots,\^x_n)\^y_i)^w=h_if_i,
\quad
\^y_i^w=f_i,
\qqbox{and}
\^z_i^w=u_i,
$$
i.e. $\^x_i^w\in H$ is a solution to system (1)
in $H$ as required.

%%%%%%%%%%%%%%%%%%%%%%%%%%%%%%

%\newpage

\end{document}